\documentclass[12pt, reqno]{amsart}
\usepackage{amssymb}
\usepackage{amsrefs}
\usepackage{amsmath}
  
 \newtheorem{thm}{Theorem}[section]
 
 \newtheorem{lem}[thm]{Lemma}
 
 \theoremstyle{definition}
 
 \theoremstyle{remark}
 \newtheorem{rem}[thm]{Remark}

 \numberwithin{equation}{section}
 \allowdisplaybreaks

\newcommand{\N}{\ensuremath{\mathbb{N}}}

 \newcommand{\Lip}{\mathop{\mathrm{Lip}}\nolimits} 
  \newcommand{\Kal}{\mathop{\mathrm{Kal}}\nolimits} 
   
   \newcommand{\Cont}{\mathop{\mathcal{C}}\nolimits}

 \DeclareMathOperator*{\Ave}{Ave}

\begin{document}

\title[Average optimality in quasi-Banach spaces]
  {Optimal average approximations for functions mapping in quasi-Banach spaces}
\thanks{F. Albiac acknowledges the support of  the Spanish Ministry for Science and  Research Grants  {\it Operators, lattices, and geometry of Banach spaces}, 
reference number MTM 2012- 31286, and  {\it Structure and complexity in Banach spaces II},
reference number MTM2010-20190-C02-02.
}

\author[F. Albiac]{F. Albiac}\address{Mathematics Department\\ Universidad P\'ublica de Navarra\\
 Pamplona\\ 31006 Spain} \email{fernando.albiac@unavarra.es}
 \author[J. L. Ansorena]{J. L. Ansorena}\address{Department of Mathematics and Computer Sciences\\
Universidad de La Rioja\\
 Logro\~no\\
26004 Spain} \email{joseluis.ansorena@unirioja.es}



\begin{abstract} 
 In 1994, M.M. Popov \cite{Popov1994} showed that the fundamental theorem of calculus fails, in general, for functions mapping from a compact interval of the real line into the $\ell_p$-spaces for $0<p<1$ , and the question arose whether such a significant result  might hold in some non-locally convex spaces. In this article we completely settle the problem by proving that the fundamental theorem of calculus breaks down in the context of {\it any} non-locally convex quasi-Banach space.  Our approach introduces  the tool of Riemann-integral averages of continuous functions, and uses it to bring out to light  the  differences in behavior of their approximates in the lack of local convexity. As a by-product of our work we solve a problem raised in \cite{AlbiacAnsorenaNA} on the different types of spaces of differentiable functions with values on a quasi-Banach space.\end{abstract}

\maketitle

\section{Introduction and background}\label{intro}

Continuous maps  from a compact interval of the real line into a Banach space $X$ are Riemann-integrable. For each $f:[a,b]\to X$, the optimal behavior  of  the averages
\[
\frac{1}{t-s}\int_{s}^{t}f(u)\,du,\qquad a\le s<t\le b,
\]
in approximating $f$ locally in norm is substantiated by the fact that,  thanks to  the convexity of the space, for every point $c\in [a,b]$ we have
\[
\left\Vert\frac{1}{t-s}\int_{s}^{t}f(u)\,du-f(c)\right\Vert=\left\Vert \int_{s}^{t}
\frac{f(u)-f(c)} {t-s} \,du\right\Vert\le \max\limits_{u\in[s,t]}\Vert f(u)-f(c)\Vert,
\]
so that
\[
\lim\limits_{(s,t)\to (c,c)}\frac{1}{t-s}\int_{s}^{t}f(u)\,du=f(c).
\]
 In other words,
 the formula
\begin{equation}\label{Aveformula}
\Ave[f](s,t)=\begin{cases}{\displaystyle \frac{1}{t-s}\int_{s}^{t}f(u)\,du}, &\text{if}\;\; a\le s<t\le b,\\
f(c), &\mbox{if}\;\;a\le s=t=c\le b,\\
{\displaystyle \frac{1}{s-t}\int_{t}^{s}f(u)\,du}, &\text{if}\;\; a\le t<s\le b,
\end{cases}
\end{equation}
defines a {\it jointly continuous} function from $[a,b]\times[a,b]$ into $X$. In particular, $\Ave[f]$ is \textit{bounded}, i.e.,
\[
\sup_{a\le s,t\le b} \Vert\Ave[f](s,t)\Vert<\infty,
\]
and \textit{separately continuous}, i.e., for fixed $s_0$ and $t_0$ in $[a,b]$ we have
\[
\lim\limits_{s\to s_0}\Ave[f](s,t_0)=\lim\limits_{t\to t_0}\Ave[f](s_0,t)=\Ave[f](s_0,t_0).
\]

The study of  averages of continuous functions mapping into quasi-Banach spaces faces obstructions from the very beginning. Indeed, if $X$ is non-locally convex,  by an old result of Mazur and Orlicz \cite{MazurOrlicz1948} there exist   continuous $X$-valued functions  failing to be Riemann-integrable. Thus, to extend the problem we suppose that $f:[a,b]\to X$ is continuous and integrable, and wonder whether, with this extra assumption,  the average function defined in \eqref{Aveformula} will retain  the  optimality   it enjoys for Banach spaces.  
The answer to this question  is negative, as the authors showed in \cite{AlbiacAnsorenaNA}. 

\begin{thm}[{\cite{AlbiacAnsorenaNA}*{Theorem 1.1}}]\label{AANA1} Suppose $0<p<1$. Then there exists a continuous Riemann-integrable function $f:[0,1]\to \ell_{p}$ whose averages 
are bounded but  fail to be optimal, in the sense that
\[
\lim\limits_{s\to 1^{-}}\Ave[f](s,1)\not= \Ave[f](1,1). 
\]
\end{thm}

Delving deeper into the subject leads to  another less expected ``pathology", namely that the averages $\Ave[f](s,t)$  need not be bounded even in the case when $\Ave[f]$ is separately continuous. 

\begin{thm}[cf.\ {\cite{AlbiacAnsorenaNA}*{Theorem 4.1}}]\label{AANA2} Let $X$ be a non-locally convex quasi-Banach space. Then there exists a continuous Riemann-integrable function $f:[0,1]\to X$ whose average function $\Ave[f]$ is separately continuous  and yet
$
\sup_{0\le s,t\le 1} \Vert\Ave[f](s,t)\Vert=\infty.
$
\end{thm}

Hence, in particular this theorem yields   the existence of  continuous Riemann-integrable functions mapping into non-locally convex spaces whose averages   are separately continuous  but
 not  jointly continuous. 
 
Of course,  by a straightforward scaling argument, Theorem~\ref{AANA1} and Theorem~\ref{AANA2} can be proved for any compact interval $[a,b]$ instead of  $[0,1]$.  That being said,  for simplicity  we choose to work on the unit interval of the real line and from now on we will denote   $[0,1]$ by $I$.
 
In view of the previous results one may ask whether the optimality of $\Ave[f]$ in approximating $f$ will follow when the function $\Ave[f]$  is  simultaneously bounded and separately continuous. Equipped with the machinery we will introduce in Section 2,  in Section 3 we provide a negative answer to this question by proving the following theorem:

\begin{thm}\label{theoremboundedseparatelynotjointly} Let $X$ be a non-locally convex quasi-Banach space. Then there exists $f:I\to X$ continuous and Riemann-integrable such that $\Ave[f]$ is bounded and separately continuous but fails to be jointly continuous. 
\end{thm}

We will also be able to extend  Theorem~\ref{AANA1} by replacing the space $\ell_p$ for $p<1$ with any non-locally convex quasi-Banach space. That is,

\begin{thm}\label{theoremboundednotseparately} 
 Let $X$ be a non-locally convex quasi-Banach space. Then there exists $f:I\to X$ continuous and Riemann-integrable such that $\Ave[f]$ is bounded  but fails to be separately continuous. 
 \end{thm}

If we put
 \begin{align*}
\Cont_{R}(I,X)&=\{f:I\to X : f \text{ is continuous and Riemann-integrable}\};\\
\Cont_{jc}(I,X)&=\{f\in \Cont_{R}(I,X) :  \Ave[f] \text{ is jointly continuous}\};\\
\Cont_{sc}(I,X)&=\{f\in \Cont_{R}(I,X) : \Ave[f] \text{ is separately continuous}\};\\
\Cont_{b}(I,X)&=\{f\in \Cont_{R}(I,X) :  \Ave[f] \text{ is bounded}\},
\end{align*}
using function-theory terminology the above gives that, on the one hand, 
\[
\Cont_{jc}(I,X)\subsetneqq \Cont_{sc}(I,X)\cap \Cont_{b}(I,X) \subsetneqq \Cont_{b}(I,X)
\]
and, on the other hand,
\[
\Cont_{sc}(I,X)\cap \Cont_{b}(I,X) \subsetneqq \Cont_{sc}(I,X).
\]
In this paper we will complete our study by showing that 
\[\Cont_{b}(I,X) \cup \Cont_{sc}(I,X)\subsetneqq \Cont_{R}(I,X),\]
or, equivalently,

 \begin{thm}\label{theoremnotboundednotseparately} 
 Let $X$ be a non-locally convex quasi-Banach space. Then there exists $f:I\to X$ continuous and Riemann-integrable such that $\Ave[f]$  is neither  bounded nor separately continuous.
  \end{thm}
  
  It is worth it  pointing out that Theorem~\ref{theoremnotboundednotseparately}  was proved for  the particular case $X=\ell_p$ ($0<p<1$) in \cite{Popov1994}*{Theorem 2.1}.

Roughly speaking, the moral of this article  will be that to hope for good approximation properties by means of the Riemann integral in the setting of non-locally convex quasi-Banach spaces we need to impose more restrictions to the functions we are working with, a subject that we plan to investigate in a further publication.  

From a different perspective, the authors introduced in \cite{AlbiacAnsorenaNA} the following spaces. We will denote by $\Cont^{(1)}(I,X)$ the space of all continuously differentiable functions from a compact interval of the real line into a quasi-Banach space $X$ which vanish at $0$. 
  $
\Cont^{(1)}_{\rm Lip}(I,X)$ will consist  of all $f\in \Cont^{(1)}(I,X)$ which are Lipschitz, equipped with the quasi-norm
\begin{equation}\label{LipNorm}
\Vert f\Vert_{\Lip}=\sup_{0\le s<t\le 1}\frac{\Vert f(t)-f(s)\Vert}{t-s}.
\end{equation}
We will also consider the space
 $\Cont^{(1)}_{\Kal}(I,X)$ of all  $f\in \Cont^{(1)}_{\Lip}(I,X)$  such that the function $g:I^{2}\to X$ given by
$$ 
g(s,t)=\begin{cases}\displaystyle \frac{f(s)-f(t)}{s-t} & s\not=t\\
 f^{\prime}(t) & s=t
 \end{cases}
 $$
  is continuous. Of course, when $X$ is a Banach space, $\Cont^{(1)}_{\Kal}(I,X)=\Cont^{(1)}_{\Lip}(I,X)=\Cont^{(1)}(I,X)$. However, if $X$ is a non-locally convex quasi-Banach space then $\Cont^{(1)}_{\Lip}(I,X)\subsetneqq\Cont^{(1)}(I,X)$ as was proved in \cite{AlbiacAnsorenaNA}. As a consequence of Theorem~\ref{theoremboundedseparatelynotjointly} we can now complete the picture and state that $\Cont^{(1)}_{\Kal}(I,X)\subsetneqq\Cont^{(1)}_{\Lip}(I,X)$.

We refer to \cite{AlbiacAnsorenaNA} for additional background  and to  \cite{KPR1985,  Rolewicz1985} for the needed terminology and  notation on quasi-Banach spaces.

\section{A technique for customizing functions}\label{tayloring}

 Here and subsequently  $(X,\Vert\cdot\Vert)$ will be a real quasi-Banach space.  By the  the Aoki-Rolewicz theorem we will assume that the quasi-norm on $X$ is $p$-subadditive for some $0<p\le 1$, i.e.,
 $$
 \Vert x+y\Vert^{p}\le \Vert x\Vert^{p}+\Vert y\Vert^{p},\qquad \forall x,y\in X.
 $$ 
The proof of our  theorems in Section 3 relies essentially on 
the following construction, originally
  inspired by \cite{Popov1994}, which we adapted and extended in  \cites{AlbiacAnsorenaNA, AlbiacAnsorenaStudia}.  
    
Let  $\boldsymbol{\lambda}=(\lambda_{k})_{k=1}^{\infty}$  be a sequence of positive scalars with $\sum_{k=1}^\infty \lambda_k =1$. Consider, for
$n\in\N$, $t_n=\sum_{k=1}^n \lambda_k$, and $I_n=[t_{n-1}, t_{n})$,
 so that $(I_k)_{k=1}^\infty$ is a sequence of disjoint intervals each  one of length $\lambda_k$, whose union is $[0,1)$. 
 
 For each $k\in \N$ let $f_{I_{k}}$ be the nonnegative piecewise linear function supported on the interval $I_{k}$ having a node at the midpoint  of the interval $c_{k}=(t_{k}+t_{k-1})/2$  with $f_{I_{k}}(c_{k})=2$  and $f(t_{k-1})= f(t_{k})=0$, i.e.,
\[
f_{I_{k}}= \begin{cases}\displaystyle
\frac{4}{t_{k}-t_{k-1}}(t-t_{k-1}) &\text{if}\;\; t\in [t_{k-1}, c_{k}),\\
\displaystyle\frac{4}{t_{k}-t_{k-1}}(t-t_{k}) &\mbox{if}\;\; t\in [c_{k}, t_{k}),\\
0  &\text{otherwise}.\end{cases}
\]
Let   $\mathbf{x}=(x_{n})_{n=1}^{\infty}$ be a sequence of vectors in  $X$. 
With these ingredients we define the function $$f=f(\boldsymbol{\lambda},\mathbf{x}):I\to X$$ as
\begin{equation}\label{definitionf}
f(t)=\begin{cases} f_{I_{k}}(t) x_{k} & \text{if}\; t\in I_{k},\\
0& \text{if}\; t=1.
\end{cases}
\end{equation}
Note that  for each $s<1$ the set $f([0,s])$ lies on a finite-dimensional subspace of $X$. Hence, we can   consider the corresponding average function $\Ave[f]$ on $[0,1)\times[0,1)$ given by
\begin{equation}\label{definitionAve}
\Ave[f](s,t)=\begin{cases}{\displaystyle \frac{1}{t-s}\int_{s}^{t}f(u)\,du}, &\text{if}\;\; 0\le s<t<1,\\
f(c), &\mbox{if}\;\;0\le s=t=c < 1,\\
{\displaystyle \frac{1}{s-t}\int_{t}^{s}f(u)\,du}, &\text{if}\;\; 0\le t<s <1.
\end{cases}
\end{equation}

Clearly, $\Ave[f]$ is jointly  continuous on $[0,1)\times[0,1)$.

The next auxiliary lemma summarizes several results and ideas contained in \cites{AlbiacAnsorenaNA, AlbiacAnsorenaStudia}. 

\begin{lem}\label{en1} 
 For a given pair  $(\boldsymbol{\lambda}, \mathbf{x})$
  we have:

\begin{enumerate}
\item[(i)] The function $f=f(\boldsymbol{\lambda},\mathbf{x})$ is continuous on $I$ if and only if $x_{k}\to 0$.

\item[(ii)] Suppose that    $X$ is   $p$-convex. 
If  $(x_{k})_{k=1}^{\infty}$ is bounded  and the sequence $(\lambda_{k})_{k=1}^\infty$ verifies $\sum_{k=1}^{\infty}\lambda_{k}^{p}<\infty$, then $\sum_{k=1}^{\infty}\lambda_{k} x_k$ converges,   $f$ is Riemann-integrable on $I$, and $\int_0^1 f(u)\, du= \sum_{k=1}^{\infty}\lambda_{k} x_k$.

\item[(iii)]  The function $\Ave[f]$ is bounded on $[0,1)\times[0,1)$ if and only if there is $\mathsf{L}> 0$ so that for all integers $m,n$ with $m\le n$,
\begin{equation}\label{condition3}
\frac{\left\Vert\displaystyle \sum_{m \le k \le  
n} \lambda_k x_k\right\Vert}{\displaystyle\sum_{m \le k \le n} \lambda_k}
 \le \mathsf{L}.
\end{equation}

\item[(iv)] Suppose $x_{k}\to 0$. Then   $\Ave[f]$  extends to a  separately continuous function  on $I^{2}$  mapping the point  $(1,1)$ to $0\in X$
 if and only if $\sum_{k=1}^\infty \lambda_k x_k$ converges and
\begin{equation}\label{condition4}
\lim\limits_{n\to\infty}\frac{\displaystyle\left\Vert\sum_{k \ge n} \lambda_k x_k\right\Vert}{\displaystyle\sum_{k 
\ge n} \lambda_k}
= 0.
\end{equation}
Moreover, if $\Ave[f]$  extends to a  separately continuous function on  $I^{2}$ that maps the point $(1,1)$ to a vector $x\in X$, then $\sum_{k=1}^\infty \lambda_k x_k$  converges and
\[
\lim\limits_{n\to\infty}\frac{\displaystyle\sum_{k \ge n} \lambda_k x_k}{\displaystyle\sum_{k 
\ge n} \lambda_k}
= x.
\]

\item[(v)] $\Ave[f]$ extends to a jointly continuous function on $I^{2}$  if and only if 
\begin{equation}\label{condition5}
\lim\limits_{m,n\to\infty}
\frac{\displaystyle\left\Vert\sum_{m \le k \le  
n} \lambda_k x_k\right\Vert}{\displaystyle\sum_{m \le k \le n} \lambda_k} =0.
\end{equation}

\end{enumerate} 
\end{lem}

Proving our results  requires rigging the technique for tailoring functions exhibited in this lemma by implementing one more layer of complexity in the choice of the sequences $(x_{k})_{k=1}^{\infty}$ and $(\lambda_{k})_{k=1}^{\infty}$ when $X$ is non-locally convex. 

In what follows, the notation $(\alpha_{i})\approx (\beta_{i})$ will be used  in the regular sense  that 
$A\alpha_{i}\le \beta_{i}\le B\alpha_{i}$ for all $i$, where $A$ and $B$ are some positive constants.

Let $q$ be a
positive integer. For any
$(\mu_j)_{j=1}^q$ in $(0,\infty)$, and  $(y_j)_{j=1}^q$ in  $X$     such that  $\sum_{j=1}^q \mu_j =1$ and
$\Vert y_j\Vert\le 1$, we have
\begin{equation*}
\left\Vert\sum_{j=1}^q \mu_j y_j\right\Vert \le \left(\sum_{j=1}^q \mu_j^p\right)^{1/p}\le q^{1/p-1} .
\end{equation*}
We set
\[  C_q=\sup\left\{ \left\Vert\sum_{j=1}^q \mu_j y_j\right\Vert \, : \,  \mu_j>0, \, \sum_{j=1}^{q} \mu_j =1, \, y_j\in X, \, \Vert y_j\Vert \le 1\right\}.\]
 Clearly  $(C_q)_{q=1}^{\infty}$ is an increasing sequence and, if $X$ is not locally convex, 
 $C_q\to \infty$.  Moreover $C_q \le q^{1/p-1}$.

From our choice of $C_{q}$, for each $q$ there exist positive scalars
$(\mu_{q,j})_{j=1}^q$ and vectors $(y_{q,j})_{j=1}^q$ in $X$ such that $\sum_{j=1}^q  \mu_{q,j} =1$,
$\Vert y_{q,j}\Vert \le 1$ and  $\left\Vert\sum_{j=1}^q \mu_{q,j} y_{q,j}\right\Vert \ge C_{q}/2.$

Every natural number $k$ can be written in a unique way in the form
\begin{equation}\label{integerrepresentation}
k=\frac{2q^2+1}{2}+\varepsilon \frac{2j-1}{2},
\end{equation}
for some $q\in\N$, $\varepsilon\in\{-1,1\}$ and $1\le j\le q$.
In fact,  for a fixed $q$ we have
\begin{itemize}
\item the set $\displaystyle \left\{\frac{2q^2+1}{2}- \frac{2j-1}{2}: 1\le j\le q\right\}$  covers all the integers between $q(q-1)+1$ and $q^2$;   
\item the set $\displaystyle \left\{\frac{2q^2+1}{2}+ \frac{2j-1}{2}: 1\le j\le q\right\}$   covers all the integers between $q^2+1$ and $q(q+1)$,
\end{itemize}
so that  the numbers $\displaystyle \left\{\frac{2q^2+1}{2}+\epsilon \frac{2j-1}{2}: 1\le j\le q \right\}$  run over all the integers between $q(q-1)+1$ and $q(q+1)$.

For each $k\in \N$, let $q=q(k)$, $ j=j(k)$, and $\varepsilon=\varepsilon(k)$ uniquely determined by the representation \eqref{integerrepresentation}.  

Let $\mathbf{a}=(A_q)_{q=1}^\infty$ and $\mathbf{b}=(\beta_q)_{q=1}^\infty$ be two sequences of positive scalars with  $\sum_{q=1}^\infty \beta_q=1/2$.  We consider  $\mathbf{x}=(x_{k})_{k=1}^{\infty}$  in $X$ given 
by 
\begin{equation}\label{vectorsconstructed}
x_k= \varepsilon A_q y_{q,j},
\end{equation}
and   $\boldsymbol{\lambda}=(\lambda_{k})_{k=1}^{\infty}$ given 
by 
\begin{equation}\label{scalarsconstructed}
\lambda_k=\beta_q\mu_{q,j}.
\end{equation}
Note that 
\[
\sum_{k=1}^{\infty} \lambda_k=\sum_{q=1}^{\infty}  \sum_{j=1}^q \sum_{\varepsilon=\pm 1} \mu_{q,j}\beta_q 
=2 \sum_{q=1}^{\infty} \sum_{j=1}^q \mu_{q,j} \beta_q 
= 2\sum_{q=1}^{\infty} \beta_q
=1,
\]
so that we can define  maps $f:I\to X$
 and $\Ave[f]:[0,1)\times[0,1)\to X$  as told in (\ref{definitionf}) and (\ref{definitionAve}), respectively.
 
 Note that now the maps $f$ and $\Ave[f]$ depend also on the choice of the scalars
$\mu_{q,j}$ and the vectors $y_{q,j}$.   In order  to avoid cumbrous notations and  to emphasize the  role  of the sequences
 $\mathbf{a}$ and $\mathbf{b}$ in the construction of  the specific functions  that will best suit our  needs  in  Section 3,  in the next lemma we will refer to  them as $f_{\mathbf{a}, \mathbf{b}}$  and 
 $\Ave[f_{\mathbf{a}, \mathbf{b}}]$.
 
 \begin{lem}\label{en2} 
 For a given pair  $(\mathbf{a}, \mathbf{b})$ with
 $\mathbf{a}=(A_q)_{q=1}^\infty$ and $\mathbf{b}=(\beta_q)_{q=1}^\infty$,
   we have:
\begin{enumerate}
\item[(i)] If $A_{q}\to 0$, then $f_{\mathbf{a}, \mathbf{b}}$ is continuous on $I$.

\item[(ii)] Suppose that    $X$ is   $p$-convex   for some $0<p\le 1$. 
If  $(A_q)_{q=1}^\infty$ is bounded  and 
$\sum_{q=1}^{\infty} q^{1-p}  \beta_{q}^{p}<\infty$, then 
$f_{\mathbf{a}, \mathbf{b}}$ is Riemann-integrable on $I$ and $\int_0^1 f_{\mathbf{a}, \mathbf{b}}(u)\, du=0$.

\item[(iii)] The average function $\Ave[f_{\mathbf{a}, \mathbf{b}}]$ is bounded on $[0,1)\times[0,1)$ if and only if
the sequence $(A_q C_q)_{q=1}^\infty$ is bounded. 

\item[(iv)] Suppose $A_{q}\to 0$. Then $\Ave[f_{\mathbf{a}, \mathbf{b}}]$  extends to a  separately continuous function on  $I^{2}$ if and only if 
\begin{equation*}
\lim\limits_{q\to\infty}\frac{ A_q C_q \beta_q}{\displaystyle\sum_{r\ge q} \beta_r}= 0.
\end{equation*}
Moreover,  in the positive case, the extended function maps the point $(1,1)$ to $0\in X$.

\item[(v)] $\Ave[f_{\mathbf{a}, \mathbf{b}}]$ extends to a jointly continuous function on $I^{2}$ if and only if 
$A_q C_q\to 0$.
\end{enumerate} 
\end{lem}

\begin{proof} Let $(\lambda_k)_{k=1}^\infty$ and $(x_k)_{k=1}^\infty$ as defined in (\ref{scalarsconstructed}) and (\ref{vectorsconstructed}) respectively.

(i) is clear from  Lemma~\ref{en1}(i).

(ii) follows from  Lemma~\ref{en1}(ii) since
\[
\sum_{k=1}^{\infty}  \lambda_k^p=\sum_{q=1}^{\infty}  \sum_{j=1}^q \sum_{\varepsilon=\pm 1} \mu_{q,j}^p\beta_q^p
= 2 \sum_{q=1}^{\infty} \sum_{j=1}^q \mu_{q,j}^p \beta_q^p
= 2 \sum_{q=1}^{\infty} q^{1-p} \beta_q^p<\infty.
\]

To show the other statements  we need to estimate $\Vert\sum_{m\le k\le n} \lambda_k x_k\Vert$. First, we consider the case
 $q(m)=q(n)=q$. 
 
 Suppose that  $m=(2q^2+1)/2- (2j_0-1)/2$ and $n=(2q^2+1)/2 + (2j_1-1)/2$. Then,
  \begin{align*}
 \sum_{k=m}^n \lambda_k x_k&=\sum_{i=1}^{\min\{j_0,j_1\}}  \sum_{\varepsilon=\pm 1}\varepsilon A_q \beta_q \mu_{q,j} y_{q,j}
 \pm\sum_{i=1+\min\{j_0,j_1\}}^{\max\{j_0,j_1\} }A_q \beta_q \mu_{q,j} y_{q,j}
\\& =
 \pm  A_q \beta_q \sum_{i=1+\min\{j_0,j_1\}}^{\max\{j_0,j_1\} } \mu_{q,j} y_{q,j}.
 \end{align*}
 In particular, if  $j_0=j_1$, $ \sum_{k=m}^n \lambda_k x_k=0$. Therefore,
  \begin{equation}\label{fundamentalvanishing}
\sum_{k=q(q-1)+1}^{q(q+1)} \lambda_k x_k=0.
\end{equation}
Also,
\begin{equation}\label{fundamentalfact}
\left\Vert\sum_{k=q^2+1}^{q(q+1)} \lambda_k x_k\right\Vert\ge A_qD_q  \beta_q.
\end{equation}

 Suppose that  $m=(2q^2+1)/2- (2j_1-1)/2$,  $n=(2q^2+1)/2 - (2j_0-1)/2$ (or 
  $m=(2q^2+1)/2 + (2j_0-1)/2$,  $n=(2q^2+1)/2 + (2j_1-1)/2$), with $j_0\le j_1$. Then,
  \[
 \sum_{k=m}^n \lambda_k x_k = \pm A_q\beta_q \sum_{i=j_0}^{j_1}\mu_{q,j} y_{q,j}.
  \]
 Since for any $1\le i_0\le i_1\le q$ we have
\[
\left\Vert\sum_{i=i_0}^{i_1}\mu_{q,j} y_{q,j}\right\Vert\le C_{i_1-i_0+1} \sum_{i=i_0}^{i_1}\mu_{q,j} 
\le C_q \sum_{i=i_0}^{i_1}\mu_{q,j},
\]
we get that, in any case, for $m$ and $n$  such that $q(m)=q(n)$,
\begin{equation}\label{fundamentalbound}
\left\Vert\sum_{k=m}^n \lambda_k x_k\right\Vert \le A_q C_q \sum_{k=m}^n \lambda_k,
\end{equation}
and
\begin{equation}\label{bound}
\left\Vert\sum_{k=m}^n \lambda_k x_k\right\Vert \le A_q C_q \beta_q.
\end{equation}

Now we consider the case $q_0=q(m)<q(n)=q_1$.  Using (\ref{fundamentalvanishing}), the $p$-subadditivity of the quasi-norm, and (\ref{fundamentalbound}),
 \begin{align*}
\left\Vert\sum_{k=m}^n \lambda_k x_k\right\Vert^p&=\left\Vert\sum_{k=m}^{q_0(q_0+1)} \lambda_k x_k + \sum_{k=q_1(q_1-1)+1}^n \lambda_k x_k \right\Vert^p
\\& \le\left\Vert\sum_{k=m}^{q_0(q_0+1)} \lambda_k x_k\right\Vert^p+\left\Vert\sum_{k=q_1(q_1-1)+1}^n \lambda_k x_k \right\Vert^p
\\& \le\left(A_{q_0}C_{q_0}\sum_{k=m}^{q_0(q_0+1)} \lambda_k\right)^p+\left(A_{q_1}C_{q_1}\sum_{k=q_1(q_1-1)+1}^n \lambda_k \right)^p 
\\&\le 2^{1-p} \left(A_{q_0}C_{q_0}\sum_{k=m}^{q_0(q_0+1)} \lambda_k+ A_{q_1}C_{q_1}\sum_{k=q_1(q_1-1)+1}^n \lambda_k  \right)^p.
\end{align*}
From this estimate we obtain 
\begin{equation}\label{finalfundamentalbound}
\left\Vert\sum_{k=m}^n \lambda_k x_k\right\Vert \le 2^{{1/p}-1} \max\{A_{q_0}C_{q_0}, A_{q_1}C_{q_1} \}  \sum_{k=m}^n \lambda_k,
\end{equation}
and
\begin{equation}\label{anotherbound}
\left\Vert\sum_{k=m}^n \lambda_k x_k\right\Vert  \le 2^{1/p}   (A_{q_0}C_{q_0}\beta_{q_0}+ A_{q_1}C_{q_1} \beta_{q_1}).
\end{equation}
Notice that both inequalities  are true even if $q(n)=q(m)$.

Now we are ready to tackle (iii). Suppose that $\Ave[f_{\mathbf{a}, \mathbf{b}}]$ is bounded. Using (\ref{condition3}) with $m=q^2+1$, $n=q(q+1)$ in combination with (\ref{fundamentalfact}) 
 we obtain that
 $(A_q C_q)_{q=1}^\infty$ is bounded. 
  
  Conversely, assume that there is a positive constant $\mathsf{M}$ such that
 $A_q C_q\le \mathsf{M}$ for all $q\in\N$.  Let $n,m\in\N$ with $m\le n$. Denote $q_0=q(m)$, $q_1=q(n)$.
Inequality
(\ref{finalfundamentalbound}) yields 
 \begin{equation*}
\left\Vert\sum_{k=m}^n \lambda_k x_k\right\Vert \le  2^{{1/p}-1} \mathsf{M} \sum_{k=m}^{n} \lambda_k.
\end{equation*}

To show (iv), assume that $\Ave[f_{\mathbf{a}, \mathbf{b}}]$ extends to a separately continuous function from $I^{2}$ into $X$. By (\ref{fundamentalvanishing})  we have that
$\sum_{r=q(q-1)+1}^\infty \lambda_k x_k=0$. Hence, the extension must send the point $(1,1)$ to $0$ and so (\ref{condition4}) holds.
Putting $n=q^2+1$  and appealing to (\ref{fundamentalvanishing}) and (\ref{fundamentalfact})  we obtain
\[
\lim_q\frac{A_q C_q \beta_q}{2(\beta_q+2\sum_{r>q}\beta_r)}=0,
\]
which yields the desired conclusion. 

Let us show the converse.  To see that $\sum_{k=1}^\infty \lambda_k x_k$ is a Cauchy series,
 note that $A_q C_q \beta_q \to 0$ and  use  (\ref{anotherbound}).
Let $n\in\N$ and consider $q=q(n)$. Combining (\ref{fundamentalvanishing})  and (\ref{bound}),  and taking into account  that the functions of the form $t\mapsto \frac{t}{t+a}$ ($a>0$) are increasing on $(0,\infty)$,
\begin{align*}
\frac{\displaystyle\left\Vert\sum_{k \ge n} \lambda_k x_k\right\Vert}{\displaystyle\sum_{k \ge n} \lambda_k}
=\frac{\displaystyle\left\Vert\sum_{k=n}^{q(q+1)} \lambda_k x_k\right\Vert}{\displaystyle\sum_{k \ge n} \lambda_k}
&\le \frac{\displaystyle A_q C_q\sum_{k=n}^{q(q+1)} \lambda_k}{\displaystyle\sum_{k \ge n} \lambda_k}
\\&= \frac{\displaystyle A_q C_q\sum_{k=n}^{q(q+1)} \lambda_k}{\displaystyle\sum_{k=n}^{q(q+1)} \lambda_k + 2\sum_{r> q} \beta_r} 
\\ &\le \frac{\displaystyle A_q C_q\beta_q}{\displaystyle\sum_{r\ge q} \beta_r} \to 0.
\end{align*}

(v)  Suppose that  $\Ave[f_{\mathbf{a}, \mathbf{b}}]$ extends to a jointly continuous function on $I^2$. Using
 (\ref{condition5})  with $m=q^2+1$ and $n=q(q+1)$ and appealing to (\ref{fundamentalfact})  we obtain $A_q C_q\to 0$. The converse follows from
 (\ref{finalfundamentalbound}).
\end{proof}

 \section{Completion of the proofs of the main theorems}\label{proofs}
 
  Throughout this section $X$ will be a non-locally convex quasi-Banach space whose quasi-norm is assumed to be $p$-subadditive for some  $p<1$. 
   
 \begin{proof}[Proof of Theorem~\ref{theoremboundedseparatelynotjointly}] 
 Pick  $b>2(1-p)/p$ and let $\mathbf{a}=(A_q)_{q=1}^\infty$ and $\mathbf{b}=(\beta_q)_{q=1}^\infty$
 be scalar sequences given by
 \begin{equation*}
 A_q=\frac{1}{C_q},\quad
 \beta_q =\frac12\left(\frac{1}{q^{b}} -\frac{1}{(1+q)^{b}}\right),\qquad q=1,2,\dots,
 \end{equation*}
 where the sequence $(C_q)_{q=1}^\infty$ is defined  in the forerunners  of  Lemma~\ref{en2}. Recall that 
 $C_q\to \infty$ because $X$ is non-locally convex.

  Notice that
\[
\sum_{q=1}^{\infty} \beta_q
=\frac12\left( 1-\lim_q \frac{1}{(1+q)^{b}}\right) =\frac12,
 \]
so that we can  construct   maps $f_{\mathbf{a}, \mathbf{b}}$
 and $\Ave[f_{\mathbf{a}, \mathbf{b}}]$ as in Section~\ref{tayloring}.
 
 Since $A_q\to 0$, $f_{\mathbf{a}, \mathbf{b}}$ is  continuous on $I$.
 That  $f_{\mathbf{a}, \mathbf{b}}$ is Riemann-integrable on $I$ follows from Lemma~\ref{en2} (ii). Indeed, taking into account that $\beta_q\approx q^{-b-1}$,
\[
\sum_{q=1}^{\infty}  q^{1-p} \beta_q^p \approx  \sum_{q=1}^{\infty}  q^{1-p}  q^{-(b+1)p}
=\sum_{q=1}^{\infty}   \frac{1}{q^{bp+2p-1}},
\]
 and this last series  converges because  $bp+2p-1> 2(1-p)+2p-1=1$.

Now, we can use  formula (\ref{Aveformula}) to  extend $\Ave[f_{\mathbf{a}, \mathbf{b}}]$ to a function $F_{\mathbf{a}, \mathbf{b}}$ on the whole square  $I^2$.  

Notice that
 \[
 \frac{A_q C_q\beta_q}{ \displaystyle \sum_{r=q}^\infty \beta_r }=
 \frac{ q^{-b}-(1+q)^{-b}   }{(q+1)^{-b}}=\left(\frac{q+1}{q}\right)^b-1\to 0.
 \]
 By Lemma~\ref{en2} (iv), $\Ave[f_{\mathbf{a}, \mathbf{b}}]$  has a separately continuous extension to $I^2$ that maps the point $(1,1)$ to $0$. This extension must be $F_{\mathbf{a}, \mathbf{b}}$.

That $F_{\mathbf{a}, \mathbf{b}}$ is bounded and discontinuous on $I^2$ follows from  Lemma~\ref{en2} (iii) and (v), respectively, since $A_q C_q=1$.
\end{proof}

\begin{proof}[Proof of Theorem~\ref{theoremboundednotseparately}]  Let $\mathbf{a}=(A_q)_{q=1}^\infty$ and $\mathbf{b}=(\beta_q)_{q=1}^\infty$ be given  by
 \begin{equation*}
 A_q=\frac{1}{C_q},\quad
 \beta_q =2^{-q-1},\qquad q=1,2,\dots
 \end{equation*}
 Notice that
$\sum_{q=1}^{\infty} \beta_q =1/2$ so that we can construct the corresponding  maps $f_{\mathbf{a}, \mathbf{b}}$
 and $\Ave[f_{\mathbf{a}, \mathbf{b}}]$.
  
 Since $A_q\to 0$,   $f_{\mathbf{a}, \mathbf{b}}$
 is  continuous on $I$. Moreover,
 \[
\sum_{q=1}^{\infty}  q^{1-p} \beta_q^p =  2^{-p}  \sum_{q=1}^{\infty}  q^{1-p} \left(2^{-p}\right)^q<\infty,
\]
 and so  $f_{\mathbf{a}, \mathbf{b}}$
 is Riemann-integrable on $I$.  Now, we can use (\ref{Aveformula}) to define a map 
 $F_{\mathbf{a}, \mathbf{b}}$ that extends $\Ave[f_{\mathbf{a}, \mathbf{b}}]$ to $I^2$. 
 
 To prove that $F_{\mathbf{a}, \mathbf{b}}$ is bounded it suffices to see that $\Ave[f_{\mathbf{a}, \mathbf{b}}]$ is bounded.
But this follows from  Lemma~\ref{en2} (iii) since $A_q C_q=1$.

We get easily
 \[
 \frac{\displaystyle A_q C_q\beta_q}{\displaystyle \sum_{r=q}^\infty \beta_r }=
 \frac{ \displaystyle 2^{-q-1}  }{ 2^{-q} }=\frac{1}{2}.
 \]
By Lemma~\ref{en2} (iv), $ \Ave[f_{\mathbf{a}, \mathbf{b}}] $ does not have  a separately continuous extension to $I^2$. Hence $F_{\mathbf{a}, \mathbf{b}}$ is not separately continuous.
\end{proof}

\begin{proof}[Proof of Theorem~\ref{theoremnotboundednotseparately}]  Let $\mathbf{a}=(A_q)_{q=1}^\infty$ and $\mathbf{b}=(\beta_q)_{q=1}^\infty$ be given  by
 \begin{equation*}
 A_q=\frac{1}{\sqrt{C_q}},\quad
 \beta_q =2^{-q-1},\qquad q=1,2,\dots,
 \end{equation*}
and construct  the maps $f_{\mathbf{a}, \mathbf{b}}$
 and  $\Ave[f_{\mathbf{a}, \mathbf{b}}]$. As in the proof of Theorem ~\ref{theoremboundednotseparately},  the function $f_{\mathbf{a}, \mathbf{b}}$ is continuous and
  Riemann-integrable on $I$, so we are able to define  $F_{\mathbf{a}, \mathbf{b}}$, which extends $\Ave[f_{\mathbf{a}, \mathbf{b}}]$ to $I^{2}$. 
 
 In order to prove that $F_{\mathbf{a}, \mathbf{b}}$ is not bounded it suffices to see that 
 $\Ave[f_{\mathbf{a}, \mathbf{b}}]$ is not bounded. This follows from  Lemma~\ref{en2} (iii) since $A_q C_q=\sqrt{C_q} \to \infty$.
 
 Finally, since
 \[
 \frac{A_q C_q\beta_q}{\displaystyle \sum_{r=q}^\infty \beta_r }=
 \frac{  2^{-q-1}  \sqrt{C_q} }{ 2^{-q} }=\frac{\sqrt{C_q}}{2} \to \infty,
 \]
 Lemma~\ref{en2} (iv) yields that $\Ave[f_{\mathbf{a}, \mathbf{b}}]$   does not have  a separately continuous extension to $I^{2}$. Hence $F_{\mathbf{a}, \mathbf{b}}$ is not separately continuous.
\end{proof}

\begin{rem} Note that both Theorem~\ref{theoremboundednotseparately} and Theorem~\ref{theoremnotboundednotseparately} yield that given a non-locally convex quasi-Banach space $X$, there exists a continuous function $f:I\to X$ whose integral function $t\mapsto \int_{0}^{t}f(u)\,du$ fails to have a left derivative at $1$. That is, the fundamental theorem of calculus breaks down, not only for $\ell_p$ when $p<1$ as was showed by Popov in the aforementioned \cite{Popov1994}*{Theorem 2.1}, but for {\it any} non-locally convex quasi-Banach space!

\end{rem}

\begin{rem} The alert reader might wonder whether it is possible to define an integral for quasi-Banach spaces that interacts well with differentiation, in the sense that the fundamental theorem of calculus remains true.  Vogt introduced in 1967 a concept of integrability quite different from that of Riemann specifically designed for $p$-Banach spaces with $p<1$. We refer to \cite{Vogt1967} for details. As it happens, it has been recently shown in \cite{AlbiacAnsorenaJFA} that the Lebesgue differentiation theorem does hold for functions mapping in quasi-Banach spaces that are integrable in the sense of Vogt.

\end{rem}

\end{document}